\documentclass[12pt]{article}

\usepackage[margin=1.05in]{geometry}
\usepackage{amsmath,amssymb,amsthm}
\usepackage{mathrsfs}
\usepackage{graphicx}
\usepackage{url}
\usepackage{hyperref}
\usepackage{enumerate}
\usepackage{authblk}
\usepackage{xcolor}

\newtheorem{theorem}{Theorem}[section]

\newtheorem{lemma}[theorem]{Lemma}
\newtheorem{corollary}[theorem]{Corollary}
\newtheorem{conjecture}[theorem]{Conjecture}

\theoremstyle{definition}

\theoremstyle{remark}
\newtheorem{remark}[theorem]{Remark}

\newcommand{\R}{\mathbb R}
\newcommand{\Z}{\mathbb Z}
\newcommand{\E}{\mathbb E}
\newcommand{\Var}{\operatorname{Var}}
\newcommand{\Cov}{\operatorname{Cov}}
\newcommand{\Tr}{\operatorname{Tr}}
\newcommand{\Id}{\operatorname{Id}}
\newcommand{\HS}{\mathrm{HS}}

\newcommand{\dd}{\,\mathrm d}
\newcommand{\ip}[2]{\left\langle #1,#2\right\rangle}
\newcommand{\norm}[1]{\left\lVert #1\right\rVert}

\newcommand{\tauM}{\tau_M}

\title{Optimal MM* bounds for convex bodies}

\author{Pierre Bizeul}

\date{}

\begin{document}

\maketitle

\begin{abstract}
Let $ K\subset\R^n$ be a convex body in isotropic position. We prove the
optimal mean-width estimate
$$
  M^*(K)\leq C\sqrt{n\log n}.
$$
The main new ingredient is a geometric inequality relating the Gaussian
mean of the support function to its mean under the uniform measure on
$K$, obtained through a heat-flow argument. Combined with the
Gaussian-log-concave comparison of Eldan and Lehec and the newly available
dimension-free bound on the third-moment parameter $\kappa_n$, this
yields the result. The boundedness of $\kappa_n$ also makes the mean-norm
estimate of Bizeul and Klartag sharp. Combining both estimates yields
$$
  M(K)M^*(K)\leq C\log n,
$$
extending Pisier's $MM^*$ estimate to non-symmetric convex bodies and to the isotropic position.
\end{abstract}

\section{Introduction}

Throughout the paper, we work in $\R^n$ with $n\geq2$. A convex body
$K\subset\R^n$ is a compact convex subset with non-empty interior. We
shall denote by $X_K$ the random vector uniformly distributed in $K$.
More generally, we say that a random vector $X$ is log-concave if it
has a density of the form $e^{-V}$, with
$V:\R^n\longrightarrow\R\cup\{+\infty\}$ a convex function. We say
that a random vector $X$ is isotropic if it is centered and its
covariance matrix is the identity. A convex body $K$ will be called
isotropic if the random vector $X_K$ is. Namely,
\[
  \int_K x\dd x=0,
  \qquad
  \frac{1}{|K|}\int_K x\otimes x\dd x=\Id,
\]
where $|\cdot|$ denotes Lebesgue volume. Note that this convention,
which we retain throughout the paper, is sometimes called
probabilistic isotropy and is slightly different from volume
isotropy, which requires the volume of $K$ to be one and its covariance
matrix to be scalar. Any convex body can be made isotropic after
applying a suitable invertible affine transformation. In the
terminology of convex geometry, an invertible affine image of a convex
body is sometimes called a ``position''.

The support function of $K$ is defined by
\[
  h_K(x)=\sup_{y\in K}\langle x,y\rangle,
  \qquad x\in\R^n.
\]
If $K$ contains the origin in its interior, we define its gauge by
\[
  \norm{x}_K
  =\inf\left\{t>0:\ \frac{x}{t}\in K\right\}.
\]
Under the same assumption, the polar body of $K$ is defined as the
body whose gauge is $h_K$, namely
\[
  K^\circ
  =
  \left\{
    x\in\R^n:\
    \sup_{y\in K}\langle x,y\rangle\leq1
  \right\}.
\]
Note that $\norm{\cdot}_K$ and $h_K$ define genuine norms only when
$K$ is origin-symmetric.

The mean-norm and mean-width are two fundamental geometric parameters
associated with a convex body $K$, defined respectively by
\[
  M(K)=\E\norm{\Theta}_K,
  \qquad
  M^*(K)=\E h_K(\Theta),
\]
where $\Theta$ is distributed uniformly on the sphere
$\mathbb S^{n-1}$. These two quantities are dual to one another since,
tautologically,
\[
  M(K^\circ)=M^*(K).
\]

It is not too difficult to provide lower bounds for $M$ and $M^*$.
Indeed, for a fixed volume, Urysohn's inequality states that the
mean-width is minimized by the Euclidean ball:
\begin{equation*}
  M^*(K)
  \geq
  \left(\frac{|K|}{|B_2^n|}\right)^{1/n} := \operatorname{vrad}(K).
\end{equation*}
Here and below, $B_p^n$ denotes the unit ball of $\ell_p^n$.
Furthermore, when $K$ contains the origin in its interior, expressing
its volume in polar coordinates and using Jensen's inequality gives
the corresponding estimate
\begin{equation*}
  M(K)
  \geq
  \left(\frac{|B_2^n|}{|K|}\right)^{1/n}.
\end{equation*}
In particular, the scale-invariant quantity
\[
  J(K):=M(K)M^*(K)
\]
satisfies
\[
  J(K)\geq1.
\]
Clearly the quantities $M(K),M^*(K)$ and $J(K)$ can be made
arbitrarily large by volume-preserving linear transformations.
The
question of upper-bounding these quantities in an appropriate position
is central to convex geometry.

The main result of this paper is the following optimal estimate on the
mean-width in isotropic position.

\begin{theorem}\label{thm_186}
Let $K$ be an isotropic convex body in $\R^n$. Then
\begin{equation}\label{eq_188}
  M^*(K)\leq C\sqrt{n\log n},
\end{equation}
where $C>0$ is a universal constant.
\end{theorem}

Theorem~\ref{thm_186} is optimal, up to the value of the
constant $C$, as follows from the example of the cross-polytope. Before
the present work, the best known estimate in isotropic position was due
to Emanuel Milman~\cite{Milman}:
\[
  M^*(K)\lesssim\sqrt n\,\log^2 n.
\]
Earlier estimates were obtained in
\cite{Hartzoulaki,Pivovarov}; as observed in
\cite{BGVV_book}, they may also be deduced
from Paouris's work~\cite{Paouris}.

After the first version of this paper appeared, Brazitikos and Pandis~\cite{BrazitikosPandis} obtained,
for origin-symmetric isotropic bodies, the intermediate estimates
\[
  M(K)\lesssim\frac{\log n}{\sqrt n},
  \qquad
  M^*(K)\lesssim\sqrt n\,\log n,
\]
using different methods, based on centroid bodies.

Finally, in independent and simultaneous work, Paouris and
Pathak~\cite{PaourisPathak} proved sharp estimates for $M$ and $M^*$ using
Bobkov's maximal Gaussian measure position~\cite{BobkovMPosition}.

As explained below, the result of Bizeul and
Klartag~\cite{BK_distance} also yields the optimal affine
bound on $M^*$ by applying the $M$-estimate to the polar body
in isotropic position and using the Bourgain-Milman reverse
Santal\'o inequality. Alternatively, Theorem~\ref{thm_186}
gives the same affine bound directly from the isotropic
position, using the slicing theorem. 
Both arguments rely on the recently established dimension-free
bound on $\kappa_n$, recalled below.

The proof of Theorem~\ref{thm_186} combines two ingredients.
The first is the following optimal comparison between Gaussian and
log-concave vectors in the gauge order.

\begin{theorem}\label{thm_206}
Let $X$ be an isotropic log-concave random vector in $\R^n$, and let
$\norm{\cdot}$ be a gauge. Then
\begin{equation*}
  \frac{\E\norm{G}}{C\sqrt{\log n}}
  \leq \E\norm{X}
  \leq C\sqrt{\log n}\,\E\norm{G},
\end{equation*}
where $G\sim N(0,\Id)$ is a standard Gaussian vector and $C>0$ is a
universal constant.
\end{theorem}

The left-hand side of this inequality was established by Bizeul and
Klartag~\cite{BK_distance}, while the right-hand side was proved by
Eldan and Lehec~\cite{EldanLehec}. In both cases, the original
statements involved an additional parameter $\kappa_n$, related to
the third-order moment tensor of isotropic log-concave vectors. 
More precisely, this parameter is defined by
\begin{equation}\label{eq_381}
  \kappa_n
  =
  \sup_X\sup_{\theta\in\mathbb S^{n-1}}
  \norm{
    \E\left[\langle X,\theta\rangle X\otimes X\right]
  }_{\HS},
\end{equation}
where the first supremum runs over all isotropic log-concave
random vectors $X$ in $\R^n$. The dimension-free bound
$\kappa_n^2\leq8$ is recorded in~\cite{Letwin}.
For completeness, Appendix~\ref{sec_513} explains how a minor
anisotropic adaptation of the Chen-Klartag argument
\cite{ChenKlartag} yields this bound.

The second ingredient is the following new inequality.

\begin{theorem}\label{thm_230}
Let $K$ be a centered convex body in $\R^n$. Then
\begin{equation*}
  \left(\E h_K(G)\right)^2
  \leq
  2n\,\E h_K(X_K),
\end{equation*}
where $G\sim N(0,\Id)$.
\end{theorem}

Theorem~\ref{thm_230} does not require an isotropic normalization
and admits a self-contained heat-flow proof. It is the main new
ingredient in the proof of Theorem~\ref{thm_186}. Taken
together, Theorems~\ref{thm_206} and~\ref{thm_230}
imply Theorem~\ref{thm_186}, as explained in
Section~\ref{sec_940}.

Theorem~\ref{thm_186} complements the corresponding mean-norm
estimate of Bizeul and Klartag, which, combined with the
dimension-free bound on $\kappa_n$, yields the following.

\begin{theorem}[Bizeul-Klartag~\cite{BK_distance}]
\label{thm_252}
Let $K$ be an isotropic convex body in $\R^n$. Then
\begin{equation}\label{eq_254}
  M(K)
  \leq
  C\sqrt{\frac{\log n}{n}},
\end{equation}
where $C>0$ is a universal constant.
\end{theorem}

The bound is sharp, as follows from the example of the cube. Prior to \cite{BK_distance}, mean-norm estimates in isotropic position were obtained in 
\cite{GiannopoulosMilman2,GSTV, Vritsiou}.

As mentioned previously, Theorem~\ref{thm_252} already yields the
sharp affine bounds for $M$ and $M^*$. Indeed, after the
normalization $|K|=|B_2^n|$, it gives an affine position in which
\[
  M(K)\lesssim\sqrt{\log n}.
\]
Applying the same estimate to the polar body in isotropic
position, and using the Bourgain-Milman reverse Santal\'o
inequality~\cite{BourgainMilman,AGAbook},
yields, after rescaling, another affine position in which
$|K|=|B_2^n|$ and
\[
  M^*(K)\lesssim\sqrt{\log n}.
\]
Thus the optimal affine estimates for $M$ and $M^*$ already follow from
the Bizeul-Klartag estimate together with the dimension-free bound on
$\kappa_n$. The
only volume estimate required is the elementary fact that an isotropic body $K$ satisfies
$\operatorname{vrad}(K)\lesssim\sqrt n$, which is the ``easy'' direction of the slicing problem. The solution of the slicing problem (\cite{KL_slicing, B_slicing} and the very recent \cite{BrazitikosSlicing}) in fact implies that $\operatorname{vrad}(K)\simeq\sqrt n$. This allows to show that the $M^*$ bound in isotropic position of Theorem \ref{thm_186} yields an affine optimal bound. See the remark after the proof of
Corollary~\ref{cor_319}.

The new point of Theorem~\ref{thm_186} is the
optimal bound for $M^*$ in the isotropic position. Together with
Theorem~\ref{thm_252}, this yields an optimal isotropic $MM^*$ estimate.

\begin{corollary}\label{cor_265}
Let $K$ be an isotropic convex body in $\R^n$. Then
\begin{equation*}
  J(K)=M(K)M^*(K)\leq C\log n.
\end{equation*}
\end{corollary}

This should be compared with the landmark result of Pisier (see \cite{Pisier}), which
states that for every origin-symmetric convex body $K\subset\R^n$,
there exists an invertible linear transformation $T\in GL_n$ such
that
\begin{equation*}
  J(TK)\leq C\log n.
\end{equation*}
Corollary~\ref{cor_265} extends this estimate to
arbitrary convex bodies, without any symmetry assumption, and shows
that one may choose the position to be isotropic.
Before the present work, $MM^*$ estimates in the non-symmetric setting were pioneered by Rudelson \cite{Rudelson} with a bound of order $n^{1/3}$, up to logarithmic factors. More recently,
Bizeul and Klartag established an optimal mean-norm bound, up to the value of $\kappa_n$, and combined it with Emanuel Milman's mean-width estimate to obtain $J(K) \leq C\log^3 n$
in isotropic position.

\paragraph{Optimality of the estimates.}

The three estimates
\[
  M(K)\lesssim\sqrt{\frac{\log n}{n}},
  \qquad
  M^*(K)\lesssim\sqrt{n\log n},
  \qquad
  J(K)\lesssim\log n
\]
are all optimal in isotropic position. This remains true even if one
restricts attention to origin-symmetric convex bodies. As mentioned
above, the cube and the cross-polytope saturate the first two
inequalities, respectively. The third one is saturated by a product of
a cube and a cross-polytope.
Details are provided in Section~\ref{sec_1045}.

For the individual quantities $M$ and $M^*$, a stronger statement holds. The estimates are in fact affine optimal: they cannot be
improved by choosing a different affine position.

\begin{corollary}\label{cor_319}
The mean-width and mean-norm estimates are affine optimal. Namely
\begin{equation}\label{eq_321}
  \sup_K\inf_{T,a} M(TK+a)
  \simeq\sqrt{\log n},
  \qquad
  \sup_K\inf_{T,a} M^*(TK+a)
  \simeq\sqrt{\log n}.
\end{equation}
Here, the suprema run over all convex bodies $K\subset\R^n$ satisfying
$|K|=|B_2^n|$, and the infima run over all $T\in SL_n$ and
$a\in\R^n$ such that $0\in\operatorname{int}(TK+a)$. The same
conclusions remain valid when the suprema are restricted to
origin-symmetric convex bodies.

\end{corollary}

The situation for the product $J(K)$ is subtler. Since $J(K)$
is scale invariant, no volume normalization is required. In the class
of all convex bodies, one has

\begin{corollary}\label{cor_340}
Corollary \ref{cor_265} is affine optimal. Namely
\begin{equation}\label{eq_342}
  \sup_K\inf_{T,a} J(TK+a)
  \simeq\log n,
\end{equation}
where the supremum runs over all convex bodies $K\subset\R^n$, and the
infimum runs over all $T\in GL_n$ and $a\in\R^n$ such that
$0\in\operatorname{int}(TK+a)$.
\end{corollary}

Thus, Corollary~\ref{cor_265} is optimal even if one is allowed to
choose an arbitrary affine position. In contrast, if one restricts attention to origin-symmetric convex
bodies, it is believed that Pisier's estimate may be improved; see
e.g. \cite[Section~6.8]{AGAbook}.

\begin{conjecture}\label{conj_358}
For every origin-symmetric convex body $K\subset\R^n$, there exists an
invertible linear transformation $T\in GL_n$ such that
\begin{equation*}
  J(TK)\leq C\sqrt{\log n},
\end{equation*}
where $C>0$ is a universal constant.
\end{conjecture}

The conjectural order is the best possible, since it is saturated by
the cube, or equivalently by the cross-polytope. Since Corollary~\ref{cor_265} is optimal within the
class of origin-symmetric convex bodies, the conjecture above cannot
hold in general if one insists on using the isotropic position.

The paper is organized as follows. In Section~\ref{sec_940} we prove
Theorem \ref{thm_230} and Theorem \ref{thm_186} and discuss the optimality of our results. In Appendix~\ref{sec_513}, we recall Chen and Klartag's argument and explain its anisotropic
adaptation which yields the bound on $\kappa_n$.

\paragraph{Notation.}
Throughout the paper, $C,c>0$ denote universal constants whose values
may change from line to line. We write $A\lesssim B$ if
$A\leq CB$, and $A\simeq B$ if both $A\lesssim B$ and $B\lesssim A$.
We denote by $|x|$ the Euclidean norm of a vector, and also by
$|K|$ the volume of a set. For a matrix
$A$, we write $\norm{A}_{\HS}=\sqrt{\Tr(A^\top A)}$ for its Hilbert-Schmidt norm, and we use the notation $A\preceq B$ for the partial order on symmetric matrices. Finally,
for $x,y\in\R^n$, we set $x\otimes y=(x_i y_j)_{i,j=1}^n$.

\paragraph{Acknowledgements.} We would like to thank Boaz Klartag, Dan Mikulincer and Beatrice-Helen Vritsiou for many insightful and stimulating conversations about $MM^*$ estimates and for their feedback on an earlier version of this manuscript. We learnt about Conjecture \ref{conj_358} from Emanuel Milman. Supported by a grant from the Israel Science Foundation (ISF Grant No. 674/24).

\section{Mean-width bound and applications}\label{sec_940}

First, we recall the very standard fact that spherical integration may be replaced by Gaussian integration in the definition of the mean-norm and mean-width. Indeed if $G$ is a standard Gaussian vector in $\R^n$, then $G = |G|\Theta$, where $\Theta$ is uniformly distributed on the sphere and $|G|$ and $\Theta$ are independent. Therefore, by homogeneity, for any $K$,
\begin{equation}\label{eq_949}
  \E\norm{G}_K=\E|G|\,M(K) \simeq \sqrt{n}M(K),
  \qquad
  \E h_K(G)=\E|G|\,M^*(K) \simeq \sqrt{n}M^*(K),
\end{equation}
where we used that $\E|G|\simeq\sqrt n.$

\subsection{A heat-flow argument}

The main new ingredient for the mean-width estimate is
Theorem~\ref{thm_230}, whose proof we now give.

\begin{proof}[Proof of Theorem~\ref{thm_230}]
By a standard approximation argument, we may assume that $K$ has a smooth boundary. Let $X_K$ be uniformly distributed in the centered
convex body $K$, and write
\[
  W_K=\E h_K(X_K),
  \qquad
  W_G=\E h_K(G).
\]
Let $P_t$ be the heat semigroup with generator $\Delta/2$, and set
\[
  F(t)=\E h_K(X_K+\sqrt{t}G)
  =\frac{1}{|K|}\int_K P_t h_K(x)\dd x.
\]
By the divergence theorem,
\[
  F'(t)
  =\frac{1}{2|K|}
  \int_{\partial K}
  \ip{\nabla P_t h_K(x)}{n_K(x)}\dd S(x),
\]
where $n_K$ denotes the outer unit normal and $\dd S$ the surface measure.
The support function is differentiable almost everywhere and
$\nabla h_K(y)\in K$ at every differentiability point. Hence, by convexity of $K$, the average
\[
  \nabla P_t h_K=P_t(\nabla h_K)
\]
also takes values in $K$. For $x\in\partial K$,
\[
  \ip{\nabla P_t h_K(x)}{n_K(x)}
  \leq h_K(n_K(x))
  =\ip{x}{n_K(x)}.
\]
A second application of the divergence theorem gives
$$
  F'(t)
  \leq\frac{1}{2|K|}
  \int_{\partial K}\ip{x}{n_K(x)}\dd S(x)
  =\frac {n}{2}.
$$
Consequently,
\[
  F(t)\leq W_K+\frac{nt}{2}.
\]
On the other hand, since $X_K$ is centered, Jensen's inequality yields
\[
  F(t)
  =\E h_K(X_K+\sqrt{t}G)
  \geq\E h_K(\sqrt{t}G)
  =\sqrt{t}\,W_G.
\]
Combining both inequalities and optimizing over $t>0$, we obtain
\[
  W_G
  \leq\inf_{t>0}\left(
    \frac{W_K}{\sqrt{t}}+\frac{n\sqrt{t}}{2}
  \right)
  =\sqrt{2nW_K},
\]
which completes the proof.
\end{proof}

\subsection{Proof of Theorem~\ref{thm_186}}

We now deduce Theorem~\ref{thm_186}. Let $K$ be isotropic. Since
$h_K=\norm{\cdot}_{K^\circ}$ is a gauge,
Theorem~\ref{thm_206} gives
\[
  \E h_K(X_K)
  \leq C\sqrt{\log n}\,\E h_K(G).
\]
Combining this estimate with Theorem~\ref{thm_230}, we obtain
\[
  \left(\E h_K(G)\right)^2
  \leq Cn\sqrt{\log n}\,\E h_K(G),
\]
and hence
\[
  \E h_K(G)\leq Cn\sqrt{\log n}.
\]
Using~\eqref{eq_949}, we conclude that
\[
  M^*(K)\leq C\sqrt{n\log n},
\]
as claimed.

\subsection{Optimality of the estimates}\label{sec_1045}

We now prove the optimality statements made in the introduction. We first
consider the isotropic position and then pass to arbitrary affine positions.

\paragraph{Sharpness in isotropic position.}
Consider the convex bodies
\[
  Q_n=\sqrt3\,B_\infty^n,
  \qquad
  P_n=\sqrt{\frac{(n+1)(n+2)}2}\,B_1^n.
\]
The bodies $Q_n$ and $P_n$ are respectively the isotropic cube and the
isotropic cross-polytope. Since
$\E\norm{G}_\infty\simeq\sqrt{\log n}$, formula
\eqref{eq_949} gives
\begin{equation}\label{eq_1061}
  M(Q_n)
  \simeq\sqrt{\frac{\log n}{n}},
  \qquad
  M^*(P_n)
  \simeq\sqrt{n\log n}.
\end{equation}
Thus both the mean-norm and mean-width estimates are sharp in
isotropic position, even among origin-symmetric bodies.

We now argue that the $MM*$ estimate is also sharp in the class of symmetric isotropic convex bodies. For
simplicity, let the dimension be $2m$ and consider
\[
  K_t=Q_m\times tP_m,\qquad t>0.
\]
Let $G=(G_1,G_2)$, where $G_1,G_2$ are independent standard Gaussian vectors
in $\R^m$. We have
\[
  \norm{(x,y)}_{K_t}
  =\max\{\norm{x}_{Q_m},t^{-1}\norm{y}_{P_m}\},
  \qquad
  h_{K_t}(x,y)=h_{Q_m}(x)+t h_{P_m}(y).
\]
Moreover,
\[
  \E\norm{G_1}_{Q_m}\simeq\sqrt{\log m},
  \qquad
  \E\norm{G_2}_{P_m}\simeq1,
\]
and
\[
  \E h_{Q_m}(G_1)\simeq m,
  \qquad
  \E h_{P_m}(G_2)\simeq m\sqrt{\log m}.
\]
Therefore,
\[
  M(K_t)\simeq\frac{1}{\sqrt m}
    \max\{\sqrt{\log m},t^{-1}\},
  \qquad
  M^*(K_t)\simeq\sqrt m(1+t\sqrt{\log m}),
\]
and hence
\begin{equation}\label{eq_1104}
  J(K_t)\simeq
  \max\{\sqrt{\log m},t^{-1}\}
  (1+t\sqrt{\log m}).
\end{equation}
Taking $t=1$, the body $K_1$ is origin-symmetric and isotropic,
and $J(K_1)\simeq\log m$.

This example is nevertheless compatible with
Conjecture~\ref{conj_358}. Indeed, taking
$t=(\log m)^{-1/2}$ gives
\[
  J(K_t)\simeq\sqrt{\log m}.
\]

\paragraph{Minimal mean-width position.}
Let $\sigma$ denote the Haar probability measure on
$\mathbb S^{n-1}$. We say that a convex body $K$ is in minimal mean-width position if
$$
  M^*(K)\leq M^*(TK),
$$
for all $ T\in SL_n$. Giannopoulos and Milman proved that $K$ is in such a position if and only if
\begin{equation}\label{eq_1135}
  \int_{\mathbb S^{n-1}}h_K(u)\,u\otimes u\dd\sigma(u)
  =\frac{M^*(K)}{n}I_n.
\end{equation}

If $K$ is $1$-symmetric the identity \eqref{eq_1135} is automatically satisfied and $K$ is in minimal mean-width position. By duality, it is also in minimal mean-norm position. In particular, the standard positions of the cube and the cross-polytope
simultaneously minimize $M$ and $M^*$ at fixed volume.

When working in the symmetric class, it is enough to consider linear images of $K$, rather than affine ones. Indeed, while the mean-width is invariant by translation, the mean-norm can only increase if an origin-symmetric body is translated. For completeness, we give a short proof of this elementary fact.

For a convex body $L$ containing the origin, write
\[
  \rho_L(u)=\max\{t\geq0: tu\in L\},
\]
so that $\norm{u}_L=1/\rho_L(u)$. Let $K$ be symmetric and assume that
$0\in\operatorname{int}(K+b)$. Put
$r_+=\rho_{K+b}(u)$ and $r_-=\rho_{K+b}(-u)$. Then both
$r_+u-b\in K$ and $r_-u+b\in K$. By convexity,
\[
  \frac{r_++r_-}{2}u\in K,
\]
and therefore $r_++r_-\leq2\rho_K(u)$. Hence
\[
  \frac{1}{2}\left(\norm{u}_{K+b}+\norm{-u}_{K+b}\right)
  =\frac{1}{2}\left(\frac{1}{r_+}+\frac{1}{r_-}\right)
  \geq\frac{2}{r_++r_-}
  \geq\norm{u}_K.
\]
Integrating over the sphere yields
\begin{equation*}
  M(K+b)\geq M(K).
\end{equation*}

\paragraph{Affine optimality of $M$ and $M^*$.}
We now prove Corollary~\ref{cor_319}. Let $K$ be a convex body
with $|K|=|B_2^n|$ and let $\widetilde K = TK + a$, $T\in SL_n$ be an affine image of $K$ such that $\widetilde K$ is centered and
$$\Cov(\widetilde K) = \frac{1}{|\widetilde K|}\int_{\widetilde K}x\otimes x \,\dd x = \lambda_K^{-2} I_n$$
for some $\lambda_K>0$. The solution of the slicing
problem~\cite{KL_slicing, B_slicing, BrazitikosSlicing} says exactly that 
$$\lambda_K \simeq \sqrt{n}.$$
Applying Theorems~\ref{thm_252} and~\ref{thm_186} to the isotropic convex body $\lambda_K \widetilde K$, we obtain
\[
  M(\widetilde K)\lesssim\sqrt{\log n},
  \qquad
  M^*(\widetilde K)\lesssim\sqrt{\log n}.
\]
This proves the upper bounds in~\eqref{eq_321}.

For the reverse direction, set $\widetilde P_n$ and $\widetilde Q_n$ to be scaled versions of $P_n$ and $Q_n$ so that their volumes
are equal to $|B_2^n|$. The scaling factors are of order $n^{-1/2}$, so~\eqref{eq_1061} becomes
\[
  M(\widetilde Q_n)\simeq\sqrt{\log n},
  \qquad
  M^*(\widetilde P_n)\simeq\sqrt{\log n}.
\]
By the preceding paragraph, $\widetilde Q_n$ and $\widetilde P_n$ are in minimal mean-norm and mean-width position. This proves the corollary, including
its restriction to origin-symmetric convex bodies.

\begin{remark}
    In the affine optimality argument above, the resolution of the slicing problem is only needed to show that the $M^*$ bound in isotropic position yields an affine optimal bound. For the $M$ bound, we only need the ``easy'' direction of the slicing problem, namely that any convex body in isotropic position has a volume radius of order at most $\sqrt{n}$, which can be proved in an elementary fashion. 

    Therefore, as discussed in the introduction, one may dualize the $M$ bound to obtain an affine optimal $M^*$ bound when $K$ is in ``dual isotropic position'', that is when $K^\circ$ is isotropic. This latter argument does not require a solution to the slicing problem.
\end{remark}

\paragraph{The affine $MM^*$ problem.}
We finally prove Corollary~\ref{cor_340}. The upper bound in
\eqref{eq_342} follows by placing an arbitrary convex body in
isotropic position and applying Corollary~\ref{cor_265}.

The reverse direction was settled by Banaszczyk, Litvak, Pajor and
Szarek  \cite{BLPS}, who proved that for any simplex $\Delta$ containing the origin in its interior,
\[
  J(\Delta)\gtrsim\log n.
\]
Together with the upper bound, this proves
\eqref{eq_342}.

Finally, the conjectural $\sqrt{\log n}$ order in the origin-symmetric class is itself
optimal. Indeed, as previously discussed, the standard position of the cube $B_\infty^n$ (or equivalently, the cross polytope) simultaneously
minimizes $M$ and $M^*$ over $SL_n$. Since $J$ is scale invariant, this
implies
\[
  \inf_{T\in GL_n}J(TB_\infty^n)
  =J(B_\infty^n)
  \simeq\sqrt{\log n}.
\]
On the other hand,~\eqref{eq_1104} shows that the
isotropic position may have an $MM^*$ product of order $\log n$, even in the
origin-symmetric class. Therefore, a proof of Conjecture \ref{conj_358} should not use the isotropic position.

\subsection{Applications of the MM* estimate}

We conclude with two consequences of the bounds on $M$ and $M^*$.
Recall that the Banach-Mazur distance between two convex bodies
$K,L\subset\R^n$ is defined by
\[
  d_{BM}(K,L)
  =\inf\left\{\lambda\geq1:\ \exists T\in GL_n,\ x,y\in\R^n,
  \ TK+x\subset L+y\subset\lambda(TK+x)\right\}.
\]
The argument of~\cite{BK_distance} gives, after putting
$K,L\subset\R^n$ in isotropic position,
\[
  d_{BM}(K,L)
  \leq C\left(
  n\max\{M(K),M(L)\}
  +\max\{M^*(K),M^*(L)\}
  \right)^2
\]
for all $n\geq2$. Together with~\eqref{eq_188} and
\eqref{eq_254}, this yields
\begin{equation*}
  d_{BM}(K,L)\leq Cn\log n.
\end{equation*}
The relevance of $MM^*$ bounds to Banach-Mazur distances was pioneered by Rudelson \cite{Rudelson}.

The same estimates also improve the flatness bound. Recall that the
flatness constant $\operatorname{Flt}(n)$ is the smallest $F$ such
that every convex body $K\subset\R^n$ satisfying
$K\cap\Z^n=\varnothing$ has lattice width at most $F$, that is,
\[
  \min_{z\in\Z^n\setminus\{0\}}
  \left(
  \sup_{x\in K}\langle x,z\rangle
  -\inf_{x\in K}\langle x,z\rangle
  \right)
  \leq F.
\]
It follows from the $MM^*$ estimate and the argument of Banaszczyk,
Litvak, Pajor and Szarek~\cite{BLPS} that
\begin{equation*}
  \operatorname{Flt}(n)\leq Cn\log n.
\end{equation*}
This improves the bound $O(n\log^2 n)$ obtained by Reis and
Rothvoss~\cite{ReisRothvoss}.

\appendix
\section{An anisotropic version of the Chen-Klartag argument}
\label{sec_513}

We stress that this appendix is purely expository and contains no
original content. It is based entirely on the preprint of Chen and Klartag
\cite{ChenKlartag}, and describes the minor anisotropic adaptation of
their argument, first recorded in~\cite{Letwin}, which yields a
generalized thin-shell estimate and, consequently, a dimension-free bound on $\kappa_n$.

The thin-shell problem originates in the work of Anttila, Ball and
Perissinaki~\cite{ABP} and was subsequently
formulated as the variance conjecture by Bobkov and
Koldobsky~\cite{BK_clt}. It asks whether every isotropic
log-concave random vector $X$ in $\R^n$ satisfies
\[
  \Var(|X|^2)\leq Cn.
\]
After a long sequence of developments, this conjecture was 
resolved by Klartag and Lehec~\cite{KlartagLehec}, following a
breakthrough of Guan~\cite{Guan}. Recently,
Chen and Klartag~\cite{ChenKlartag} gave a different proof, based on
log-concave moment measures, and obtained the optimal constant $8$.

The purpose of this appendix is to explain the minor anisotropic
modification of their argument which gives
\begin{equation*}
  \Var(\ip{AX}{X})\leq8\Tr(A^2)
\end{equation*}
for any isotropic log-concave vector $X$ and any symmetric matrix $A$, as
recorded in~\cite{Letwin}. The question of studying thin-shell
estimates beyond the isotropic setting was put forward
in~\cite{AlonsoBastero}.

Let $X$ be a log-concave vector with law $\mu$. A common starting point in recent approaches to thin-shell problems is
the $H^{-1}$ inequality of Klartag~\cite{KlartagHminus,BartheKlartag}, which implies:
\begin{equation}\label{eq_577}
  \Var(|X|^2)
  \leq
  4\sum_{i=1}^n\norm{x_i}_{H^{-1}(\mu)}^2,
\end{equation}
where the $H^{-1}$ norm of a centered function 
$f\in L^2(\mu)$ is
\begin{equation*}
  \norm{f}_{H^{-1}(\mu)}
  =
  \sup\left\{
    \int fg\dd\mu:
    g\in C_c^\infty(\R^n),\
    \int|\nabla g|^2\dd\mu\leq1
  \right\}.
\end{equation*}

The $H^{-1}$ norm of the coordinates $x_i$ may be upper-bounded using a Stein kernel. Recall that a square-integrable matrix field
$\tau:\R^n\to\R^{n\times n}$ is a Stein kernel for $\mu$ if
\begin{equation*}
  \int x\cdot F(x)\dd\mu(x)
  =
  \int\ip{\tau(x)}{DF(x)}_{\HS}\dd\mu(x)
\end{equation*}
for every $F\in C_c^\infty(\R^n;\R^n)$, where
$DF$ denotes the Jacobian matrix of $F$.
\begin{lemma}\label{prop_594}
Let $\tau$ be a Stein kernel for $\mu$, then
\begin{equation}\label{eq_596}
  \Var(|X|^2)
  \leq
  4\int\norm{\tau(x)}_{\HS}^2\dd\mu(x).
\end{equation}
\end{lemma}

\begin{proof}
For any $g\in C_c^\infty(\R^n)$, choosing
$F=ge_i$ gives
\[
  \int x_i g\dd\mu
  =
  \int\sum_{j=1}^n\tau_{ij}\partial_jg\dd\mu.
\]
using Cauchy-Schwarz and summing over coordinates, we get
\[
  \sum_{i=1}^n\norm{x_i}_{H^{-1}(\mu)}^2
  \leq
  \int\norm{\tau(x)}_{\HS}^2\dd\mu(x).
\]
In view of~\eqref{eq_577}, this
proves~\eqref{eq_596}.
\end{proof}

The proof now consists in building a Stein kernel with controlled Hilbert-Schmidt norm using \emph{moment measures}, which we now introduce.

Let $\mu$ be a centered log-concave probability measure on $\R^n$,
with covariance matrix
\[
  \Sigma=\int x\otimes x\dd\mu(x).
\]
The moment-measure theorem of Cordero-Erausquin and
Klartag~\cite{CorderoKlartag} provides an essentially continuous
convex function $\psi:\R^n\to\R\cup\{+\infty\}$, unique up to
translation of its argument, such that
\begin{equation*}
  \int_{\R^n}e^{-\psi(y)}\dd y=1,
  \qquad
  (\nabla\psi)_\#\left(e^{-\psi(y)}\dd y\right)=\mu.
\end{equation*}

We write $d\nu(y)=e^{-\psi(y)}\,dy$ and
$d\mu(x)=e^{-V(x)}\,dx$.
Under suitable regularity assumptions, Klartag~\cite{KlartagMoment}
proved that $\psi$ is smooth and strictly convex, and that
$\nabla\psi$ is a diffeomorphism.

Let $\varphi=\psi^*$ be the Legendre transform
of $\psi$, the identities
\[
  \nabla\varphi=(\nabla\psi)^{-1},
  \qquad
  \nabla^2\varphi(\nabla\psi(y))=\nabla^2\psi(y)^{-1}
\]
show that $\nabla\psi$ is a measure-preserving isometry
between $M_\mu^* = (\R^n, \nabla^2\psi, \nu)$ and $M_\mu=(\R^n,\nabla^2\varphi,\mu)$. The central object of the proof is
\begin{equation}\label{eq_657}
  H(y)=\nabla^2\psi(y).
\end{equation}
Let $L$ be the symmetric generator associated to $M_\mu^*$, that is 
\[
  Lf
  = e^\psi\operatorname{div}\left(e^{-\psi}H^{-1}\nabla f\right)
  = \Tr(H^{-1}\nabla^2f)
    -\ip{\nabla V(\nabla\psi)}{\nabla f}.
\]
Let also $\Gamma$ be its associated carré du champ:
$$\Gamma(f,g)=\ip{H^{-1}\nabla f}{\nabla g}.$$

The Brascamp-Lieb inequality for $\nu$ implies that the space $M_\mu^*$ satisfies a Poincaré inequality with constant $1$:
\[
  \Var_\nu(f)\leq\int\Gamma(f)\dd\nu,
\]
where we abbreviated $\Gamma(f) = \Gamma(f,f)$.
By integration by parts, for any $F\in C_c^\infty(\R^n;\R^n)$,
\begin{align*}
  \int_{\R^n}x\cdot F(x)\dd\mu(x)
  &=\int_{\R^n}\nabla\psi(y)\cdot F(\nabla\psi(y))\dd\nu(y)\\
  &=\int_{\R^n}\Tr\left(H(y)DF(\nabla\psi(y))\right)\dd\nu(y)\\
  &=\int_{\R^n}\ip{\tauM(x)}{DF(x)}_{\HS}\dd\mu(x),
\end{align*}
where 
\begin{equation*}
  \tauM(x)=H\left((\nabla\psi)^{-1}(x)\right).
\end{equation*}
In other words, $\tauM$ is a Stein kernel for $\mu$, as observed by Fathi \cite{Fathi}. In particular,
\begin{equation}\label{eq_678}
  \int_{\R^n}H\dd\nu
  =\int_{\R^n}x\otimes x\dd\mu(x)
  =\Sigma.
\end{equation}

The main argument of Chen-Klartag is a bootstrap on the quantity of interest 
\begin{equation*}
  N_2=\int_{\R^n}\Tr(H^2)\dd\nu
  =\int\norm{\tauM(x)}_{\HS}^2\dd\mu(x).
\end{equation*}

\begin{theorem}[Anisotropic Chen-Klartag]
    $N_2 \leq 2\Tr(\Sigma^2)$.
\end{theorem}
Formally, only the case $\Sigma=I_n$ is proved in \cite{ChenKlartag}. Nevertheless, the proof works verbatim for an arbitrary covariance matrix $\Sigma$. Define the Dirichlet energy of $H$ by
\begin{equation*}
  D_2=\int_{\R^n}
       \sum_{i,j=1}^n\Gamma(H_{ij})\dd\nu.
\end{equation*}
Applying Brascamp-Lieb to each entry of $H$ and using
\eqref{eq_678}, we obtain

\begin{equation}\label{eq_961}
  N_2-\Tr(\Sigma^2)
  =\int\norm{H-\Sigma}_{\HS}^2\dd\nu
  \leq D_2.
\end{equation}
This is an elementary anisotropic extension of
\cite[Lemma~3.4]{ChenKlartag}. The rest of the argument is now identical to \cite{ChenKlartag}. An integration by parts gives 
\begin{equation}\label{eq_968}
  D_2=-\int_{\R^n}\Tr(HLH)\dd\nu,
\end{equation}
where $L$ acts entrywise on $H$. To see this, use the identity 
$\int fLg \,d\nu = -\int\Gamma(f,g) \,d\nu $ 
with
$f=g=H_{ij}$ and sum over $i,j$.

The key step in the proof of Chen and Klartag
is the pointwise inequality
\begin{equation}\label{eq_979}
  \sum_{i,j=1}^n\Gamma(H_{ij})
  -\Tr(HLH)
  \leq\Tr(H^2).
\end{equation}
The proof goes by differentiating the Monge-Ampère equation twice and using the symmetry of third
derivatives of $\psi$. We refer to
\cite[Lemmas~3.1 and~3.5]{ChenKlartag} for details.

Integrating \eqref{eq_979} and using
\eqref{eq_968}, we obtain
\[
  2D_2\leq N_2.
\]
Combining this with \eqref{eq_961} concludes
the bootstrap:
\[
  N_2-\Tr(\Sigma^2)\leq D_2\leq\frac{N_2}{2},
\]
and consequently
\begin{equation}\label{eq_999}
  N_2\leq2\Tr(\Sigma^2).
\end{equation}

\begin{corollary}\label{coro756}
Let $X$ be an isotropic log-concave random vector in $\R^n$.
Then, for every symmetric matrix $A$,
\begin{equation*}
  \Var(\ip{AX}{X})\leq 8\Tr(A^2).
\end{equation*}
\end{corollary}

\begin{proof}
Without loss of generality, we may assume that
$A=\operatorname{diag}(a_1,\ldots,a_n)$ for reals $a_i\neq 0$. Set
\[
  Y_i=\sqrt{|a_i|}X_i,
  \qquad
  S=\operatorname{diag}(\operatorname{sgn}(a_1),\ldots,
  \operatorname{sgn}(a_n)),
\]
Then $\ip{AX}{X}=\ip{SY}{Y}$ and $\operatorname{Cov}(Y)=\operatorname{diag}(|a_1|,\ldots,|a_n|)$. Let $\tau$ be the moment-measure Stein kernel of the log-concave vector $Y$.
Applying the $H^{-1}$ inequality to
$f(y)=\ip{Sy}{y}$, whose gradient is $2Sy$, and
using the Stein identity as in Lemma~\ref{prop_594},
we obtain
\[
\begin{aligned}
  \Var(\ip{SY}{Y})
  &\leq 4\sum_{i=1}^n
       \norm{(Sy)_i}_{H^{-1}(\mu_Y)}^2\\
  &\leq 4\int\norm{S\tau(y)}_{\HS}^2\dd\mu_Y(y)\\
  &=4\int\norm{\tau(y)}_{\HS}^2\dd\mu_Y(y)\\
  &\leq 8\Tr(\operatorname{Cov}(Y)^2)
   =8\Tr(A^2),
\end{aligned}
\]
where $\mu_Y$ is the law of $Y$ and the last inequality follows from
\eqref{eq_999}.
\end{proof}

We now return to the third-moment parameter $\kappa_n$ defined
in~\eqref{eq_381}.

\begin{corollary}\label{cor_915}
One has
\begin{equation*}
  \kappa_n^2\leq8.
\end{equation*}
\end{corollary}

\begin{proof}
Let $X$ be isotropic and log-concave, let
$\theta\in\mathbb S^{n-1}$, and set
\[
  B=\E\left[\langle X,\theta\rangle X\otimes X\right].
\]
Then $B$ is symmetric and
\begin{align*}
  \norm{B}_{\HS}^2
  &=\E\left(\langle X,\theta\rangle.
               \langle BX,X\rangle\right)\\
  &\leq \Var\left(\langle BX,X\rangle\right)^{1/2}
  \leq2\sqrt{2}\norm{B}_{\HS},
\end{align*}
where we used isotropy and Corollary \ref{coro756}. Hence
$\norm{B}_{\HS}^2\leq8$.
\end{proof}


\begin{thebibliography}{99}

\bibitem{AlonsoBastero}
D.~Alonso-Guti\'errez and J.~Bastero,
\emph{The variance conjecture on some polytopes},
in Asymptotic Geometric Analysis, Fields Inst. Commun. 68,
Springer, New York, 2013, 1-20.

\bibitem{ABP}
M.~Anttila, K.~Ball, and I.~Perissinaki,
\emph{The central limit problem for convex bodies},
Trans. Amer. Math. Soc. \textbf{355} (2003), no.~12, 4723-4735.

\bibitem{AGAbook}
S.~Artstein-Avidan, A.~Giannopoulos, and V.~D.~Milman,
\emph{Asymptotic Geometric Analysis, Part I},
Mathematical Surveys and Monographs, Vol.~202,
American Mathematical Society, Providence, RI, 2015.


\bibitem{BLPS}
W.~Banaszczyk, A.~E.~Litvak, A.~Pajor, and S.~J.~Szarek,
\emph{The flatness theorem for nonsymmetric convex bodies via the local theory of Banach spaces},
Math. Oper. Res. \textbf{24} (1999), no.~3, 728-750.

\bibitem{BartheKlartag}
F.~Barthe and B.~Klartag,
\emph{Spectral gaps, symmetries and log-concave perturbations},
Bull. Hellenic Math. Soc. \textbf{64} (2020), 1-31.

\bibitem{B_slicing}
P.~Bizeul,
\emph{The slicing conjecture via small ball estimates},
arXiv:2501.06854, 2025.

\bibitem{BK_distance}
P.~Bizeul and B.~Klartag,
\emph{Distances between non-symmetric convex bodies: optimal bounds up to polylog},
arXiv:2510.20511v3, 2026.

\bibitem{BK_clt}
S.~G.~Bobkov and A.~Koldobsky,
\emph{On the central limit property of convex bodies},
in Geometric Aspects of Functional Analysis (2001-02),
Lecture Notes in Math., Vol.~1807, Springer, 2003, 44-52.

\bibitem{BobkovMPosition}

S.~G.~Bobkov,
\emph{On Milman's ellipsoids and $M$-position of convex bodies},
in Concentration, Functional Inequalities and Isoperimetry,
Contemp. Math. \textbf{545}, Amer. Math. Soc., Providence, RI, 2011,
23-33.

\bibitem{BourgainMilman}
J.~Bourgain and V.~D.~Milman,
\emph{New volume ratio properties for convex symmetric bodies in
$\mathbb R^n$},
Invent. Math. \textbf{88} (1987), no.~2, 319-340.


\bibitem{BrazitikosSlicing}
S.~Brazitikos,
\newblock A deterministic proof of the slicing theorem,
\newblock \emph{arXiv preprint arXiv:2608.14342}, 2026.

\bibitem{BGVV_book}
S.~Brazitikos, A.~Giannopoulos, P.~Valettas, and B.-H.~Vritsiou,
\emph{Geometry of Isotropic Convex Bodies},
Mathematical Surveys and Monographs, Vol.~196,
American Mathematical Society, Providence, RI, 2014.

\bibitem{BrazitikosPandis}
S.~Brazitikos and C.~Pandis,
\emph{Geometric bounds for the mean gauge and the mean width in
isotropic position},
arXiv:2608.07216, 2026.


\bibitem{ChenKlartag}
Y.~Chen and B.~Klartag,
\emph{Digesting the proof of the sharp thin-shell inequality},
arXiv:2607.23307, 2026.

\bibitem{CorderoKlartag}
D.~Cordero-Erausquin and B.~Klartag,
\emph{Moment measures},
J. Funct. Anal. \textbf{268} (2015), no.~12, 3834-3866.

\bibitem{EldanLehec}
R.~Eldan and J.~Lehec,
\emph{Bounding the norm of a log-concave vector via thin-shell estimates},
in Geometric Aspects of Functional Analysis: Israel Seminar (GAFA) 2011-2013,
Springer, 2014, 107-122.

\bibitem{Fathi}
M.~Fathi,
\emph{Stein kernels and moment maps},
Ann. Probab. \textbf{47} (2019), no.~4, 2172-2185.

\bibitem{GiannopoulosMilman}
A.~A.~Giannopoulos and V.~D.~Milman,
\emph{Extremal problems and isotropic positions of convex bodies},
Israel J. Math. \textbf{117} (2000), 29-60.

\bibitem{GiannopoulosMilman2}
A.~Giannopoulos and E.~Milman,
\emph{$M$-estimates for isotropic convex bodies and their
$L_q$-centroid bodies},
in \emph{Geometric Aspects of Functional Analysis},
Lecture Notes in Math. \textbf{2116},
Springer, Cham, 2014, 159-182.

\bibitem{GSTV}
A.~Giannopoulos, P.~Stavrakakis, A.~Tsolomitis, and B.-H.~Vritsiou,
\emph{Geometry of the $L_q$-centroid bodies of an isotropic
log-concave measure},
Trans. Amer. Math. Soc. \textbf{367} (2015), no.~7, 4569-4593.

\bibitem{Guan}
Q.~Guan,
\emph{A note on Bourgain's slicing problem},
arXiv:2412.09075, 2024.

\bibitem{KlartagHminus}
B.~Klartag,
\emph{A Berry-Esseen type inequality for convex bodies with an unconditional basis},
Probab. Theory Related Fields \textbf{145} (2009), no.~1-2, 1-33.


\bibitem{KlartagMoment}
B.~Klartag,
\emph{Logarithmically-concave moment measures I},
in Geometric Aspects of Functional Analysis,
Lecture Notes in Math., Vol.~2116, Springer, Cham, 2014, 231-260.

\bibitem{KL_slicing}
B.~Klartag and J.~Lehec,
\emph{Affirmative resolution of Bourgain's slicing problem using Guan's bound},
Geom. Funct. Anal. \textbf{35} (2025), 1147-1168.

\bibitem{KlartagLehec}
B.~Klartag and J.~Lehec,
\emph{Thin-shell bounds via parallel coupling},
arXiv:2507.15495v2, 2026.

\bibitem{Hartzoulaki}
M.~Hartzoulaki,
\emph{Probabilistic methods in the theory of convex bodies},
Ph.D. thesis, University of Crete, March 2003.

\bibitem{Letwin}
B.~Letwin,
\emph{The KLS constant is $O(\log^{1/4}n)$},
arXiv:2607.24164, 2026.

\bibitem{Milman}
E.~Milman,
\emph{On the mean-width of isotropic convex bodies and their associated $L_p$-centroid bodies},
Int. Math. Res. Not. IMRN (2015), no.~11, 3408-3423.

\bibitem{Paouris}
G.~Paouris,
\emph{Concentration of mass on convex bodies},
Geom. Funct. Anal. \textbf{16} (2006), no.~5, 1021-1049.

\bibitem{PaourisPathak}
G.~Paouris and R.~Pathak,
\emph{Optimal mean width and metric entropy estimates for convex
bodies},
arXiv:2607.29522, 2026.

\bibitem{Pisier}
G.~Pisier,
\emph{The Volume of Convex Bodies and Banach Space Geometry},
Cambridge Tracts in Mathematics, Vol.~94,
Cambridge University Press, Cambridge, 1989.

\bibitem{Pivovarov}
P.~Pivovarov,
\emph{On the volume of caps and bounding the mean-width of an
isotropic convex body},
Math. Proc. Cambridge Philos. Soc. \textbf{149} (2010), no.~2,
317-331.

\bibitem{ReisRothvoss}
V.~Reis and T.~Rothvoss,
\emph{The subspace flatness conjecture and faster integer programming},
arXiv:2303.14605v5, 2026.

\bibitem{Rudelson}
M.~Rudelson,
\emph{Distances between non-symmetric convex bodies and the $MM^*$-estimate},
Positivity \textbf{4} (2000), no.~2, 161-178.

\bibitem{Vritsiou}
B.-H.~Vritsiou,
\emph{Regular ellipsoids and a Blaschke-Santal\'o-type inequality
for projections of non-symmetric convex bodies},
J. Funct. Anal. \textbf{286} (2024), no.~11,
Paper No.~110414, 40~pp.

\end{thebibliography}
\end{document}